\documentclass[10pt]{amsart}
\usepackage{egastyle}

\usepackage{amssymb,amsmath,amscd,amsthm,amsfonts,wasysym,mathrsfs,amsxtra,verbatim}
\input xy
\xyoption{all}

\newcommand{\ch}{{\mathrm{ch}}}
\newcommand{\dimension}{\mathrm{dim}\,}

\newcommand{\Coh}{{\mathrm{Coh}}}
\newcommand{\Ac}{{\mathcal{A}}}
\newcommand{\Tc}{{\mathcal{T}}}
\newcommand{\Stab}{\mathrm{Stab}}
\newcommand{\Aut}{\mathrm{Aut}}

\begin{document}

\title{Intersection numbers on  fibrations and Catalan Numbers}

\begin{abstract}
On an elliptic surface or threefold, Catalan numbers appear when one tries to compute the autoequivalence group action  on the Bridgeland stability manifold.  We explain why this happens by identifying a class of equations in the Chow ring of a fibration, where the solutions always involve Catalan numbers.
\end{abstract}


\author{Rimma H\"am\"al\"ainen}
\address{Institut f\"ur Mathematik \\
Technische Universit\"at Berlin\\
Str. des 17. Juni 136, 10587 Berlin\\
Germany}
\email{haemael@math.tu-berlin.de}

\author{Jason Lo}
\address{Department of Mathematics \\
California State University, Northridge\\
18111 Nordhoff Street\\
Northridge CA 91330 \\
USA}
\email{jason.lo@csun.edu}

\author{Edward Morales}
\address{Department of Mathematics \\
California State University, Northridge\\
18111 Nordhoff Street\\
Northridge CA 91330 \\
USA}
\email{edward.morales@csun.edu}


\maketitle

\tableofcontents

\section{Introduction}

The sequence  of Catalan numbers
\[
\{c_n\}_{n \geq 0} : 1, 1, 2, 5, 14, 42, 132, \cdots
\]
can be defined by  the recurrence relation
\[
  c_n = \sum_{i+j=n-1} c_ic_j \text{\quad for $n \geq 1$}
\]
and $c_0=1$, and is one of the most well-known integer sequences partly due to its appearance in a wide range of contexts.  Stanley's book \cite{stanley2015}, for examples, describes 214 different situations in which Catalan numbers arise.   In this article, we describe how Catalan numbers appear in  computational problems that arise in the area of Bridgeland stability conditions in algebraic geometry, and  identify a class of equations in the Chow ring on a fibration that explains the appearance of  Catalan numbers.  Let us begin by explaining what the computational problems are in the area of Bridgeland stability conditions.

Given any triangulated category $\Tc$, Bridgeland defined a space of stability conditions $\Stab (\Tc)$  in the paper \cite{StabBri}, which first appeared in 2002.  The space $\Stab (\Tc)$  can be considered as an invariant of the triangulated category $\Tc$.  When $\Tc$ is the  derived category of coherent sheaves on an algebraic variety, understanding the structure of $\Stab (\Tc)$ is an important problem not only in algebraic geometry, but also in representation theory and mathematical physics (e.g.\ mirror symmetry) \cite{SCK3}.  

One of the techniques for studying $\Stab (\Tc)$ is via  group actions.  The space $\Stab (\Tc)$ comes with two natural group actions: a left-action by the group $\Aut (\Tc)$ of autoequivalences of $\Tc$, and a right-action by $\widetilde{\mathrm{GL}}^+\!(2,\mathbb{R})$, the universal cover of $\mathrm{GL}^+\!(2,\mathbb{R})$.  The autoequivalence group of a triangulated category can be very complicated, such as when $\Tc$ is the derived category of coherent sheaves on a K3 surface.  As a result, it is a highly nontrivial problem to understand the $\Aut (\Tc)$-action on $\Stab (\Tc)$ completely.  In contrast, the $\widetilde{\mathrm{GL}}^+\!(2,\mathbb{R})$-action on $\Stab (\Tc)$ is very easy to understand, as it essentially just relabels certain numerical data (namely, the `phases' of semistable objects).  Hence finding solutions to  equations of the form
\begin{equation}\label{eq:stabeq}
  \Psi \cdot \sigma = \sigma' \cdot \tilde{g}
\end{equation}
where $\Psi \in \Aut (\Tc)$, $\sigma, \sigma' \in \Stab (\Tc)$ and $\tilde{g} \in \widetilde{\mathrm{GL}}^+\!(2,\mathbb{R})$ is very useful  for  understanding the space $\Stab (\Tc)$ of stability conditions on $\Tc$ \cite{Qiu, Lo20, LM2}.  Solving equations of the form \eqref{eq:stabeq} is also a fundamental problem in the theory of counting invariants \cite{toda2013gepner} and dynamical systems \cite{FAN2021107732} (the interested reader may refer to the introduction of \cite{LM2} for more details on these connections).

In practice, there are two distinct problems when it comes to solving equations of the form \eqref{eq:stabeq}: solving a central charge equation, and proving the compatibility of t-structures.  This is because an element of $\Stab (\Tc)$ is defined as a pair $\sigma = (Z,\Ac)$ where $Z : K(\Tc)\to \mathbb{C}$ is a group homomorphism from the Grothendieck group $K(\Tc)$ of $\Tc$ to the additive group of complex numbers, and $\Ac$ is the heart of a bounded t-structures on $\Tc$.  There are some compatibility conditions that $Z$ and $\Ac$ must satisfy, but they will not be relevant in this article (see \cite{StabBri} for details).  If we write $\sigma = (Z, \Ac)$ and $\sigma = (Z', \Ac')$, and let $g$ denote the projection of $\tilde{g}$ onto $\mathrm{GL}^+\!(2,\mathbb{R})$, then  the central charge equation corresponding to \eqref{eq:stabeq} takes the form
\begin{equation}\label{eq:centralchargeeq}
Z' (\Psi (-)) = g\cdot Z(-)
\end{equation}
where, on the right-hand side, we identify elements of $\mathbb{C}$ as elements of $\mathbb{R}^2$ thought of as column vectors.

In this article, we will focus on solutions to equations of the form \eqref{eq:centralchargeeq}.  In the second author's joint work with Wong \cite{LoWong}, it was observed that when $\Tc$ is the derived category of coherent sheaves on a Weierstra{\ss} elliptic surface, Catalan numbers appear in the construction of a family of solutions to  \eqref{eq:centralchargeeq}.  In Section \ref{sec:Catalanonellip3}, we show that Catalan numbers also appear when one attempts a similar construction on Weierstra{\ss} elliptic threefolds.  In Section \ref{sec:fibandCatalan}, we present a result that explains the appearance of Catalan numbers   on  elliptic surfaces and elliptic threefolds - it turns out that there is a particular class of equations in the Chow ring where the solutions always involve Catalan numbers (see \eqref{eq:Catappears}).  The first part of this  main result in Section  \ref{sec:fibandCatalan} is  as follows:

\begin{prop}[Proposition  \ref{main:ample_divisor}]\label{pro:propcatappearsparaphrase}
Suppose $p : X \to Y$ is a flat morphism of relative dimension 1, where $X$ and $Y$ are  smooth projective varieties with $X$ being $n$-dimensional, and that $p$ admits a section whose image is denoted by $\Theta$.  Suppose $\Theta^2 = h\Theta p^\ast H$ for some divisor $H$ on $Y$ and $h \in \mathbb{R}$, and that $\omega$ is a divisor on $X$ of the form $\omega = u\Theta + vp^\ast H$
for some $u,v \in \mathbb{R}$.  Then for any $0 \neq r \in \mathbb{R}$, the condition
\begin{equation}\label{eq:unifyeqparaphrase}
  \omega^2 W = r(\Theta p^\ast H) W \text{\qquad for all $W \in p^\ast(A^{n-2}(Y))$}
\end{equation}
has a power series solution in $u$ of the form
\[
  u=\sum_{i=0}^\infty c_i r^{i+1}(-h)^i\left(\tfrac{1}{2v}\right)^{2i+1}
\]
where $c_i$ is the $i$-th Catalan number.
\end{prop}

Notice that when $X$ is a surface, \eqref{eq:unifyeqparaphrase} is precisely of the form 
\[
  \omega^2 = r'
 \]
for some constant $r' \in \mathbb{R}$.  When $X$ is a threefold, however, \eqref{eq:unifyeqparaphrase} is \emph{not} of the form 
\[
\omega^3 = r''
\]
for some constant $r'' \in \mathbb{R}$.  This suggests that the fibration structure $p$ on the variety $X$ does play an essential role in the appearance of Catalan numbers when trying to solve equations of the form \eqref{eq:stabeq} on $X$.

\textbf{Organisation of the paper.} In Section \ref{sec:prelim}, we briefly recall the basics on elliptic fibrations that will be needed in this article. In particular,  Section \ref{para:prelim-nonAG} contains an informal survey of the main terminology from algebraic geometry used throughout this article, so that readers without backgrounds in algebraic geometry can also understand the computations surrounding Catalan numbers.  In Section \ref{sec:Catalanonellip3}, we describe how solutions to equations of the form \eqref{eq:centralchargeeq} on elliptic threefolds give rise to  Catalan numbers.  In Section \ref{sec:fibandCatalan}, we state our main result, Proposition \ref{main:ample_divisor}, which describes a class of equations in the Chow ring of a fibration that always gives rise to Catalan numbers.  We then explain how Proposition  \ref{main:ample_divisor}  reduces to the main result in the second author's work with Wong on elliptic surfaces \cite{LoWong}, as well as how it reduces to the computation in Section \ref{sec:Catalanonellip3} on elliptic threefolds.


\textbf{Acknowledgements.} The second and third authors were partially supported by NSF grant DMS-2100906.

\section{Preliminaries}\label{sec:prelim}

For readers who are interested in how Catalan numbers arise in the theory of stability conditions in algebraic geometry, but do not possess backgrounds in algebraic geometry, they can refer to \ref{para:prelim-nonAG}, where we explain how to decipher some of the algebraic geometry terminology in this article.

\paragraph[Notation]  For a smooth projective variety $X$, we will always write $D^b(X)$ to denote $D^b(\Coh (X))$, the bounded derived category of coherent sheaves on $X$.  For any $\mathbb{R}$-divisor $B$ on $X$, we will also write $\ch^B(E)$ to denote the twisted Chern character
\[
\ch^B(E) = e^{-B}\ch (E) 
\]
so that 
\begin{align*}
    \ch_0^B(E) &= \ch_0(E) \\
    \ch_1^B(E) &= \ch_1(E)-B\ch_0(E) \\
    \ch_2^B(E) &= \ch_2(E) - B\ch_1(E)+\tfrac{B^2}{2}\ch_0(E) \\
    \ch_3^B(E) &= \ch_3(E)-B\ch_2(E) +\tfrac{B^2}{2}\ch_1(E)-\tfrac{B^3}{6}\ch_0(E).
\end{align*}

\paragraph[Elliptic fibrations] \label{para:ellipfib} In this article, we adopt the definition of  Weierstra{\ss} elliptic fibrations in \cite[Section 6.2]{FMNT}: A Weierstra{\ss} elliptic fibration $p : X \to Y$ is a flat morphism between smooth projective varieties such that
\begin{itemize}
    \item The fibers of $p$ are Gorenstein curves of arithmetic genus 1.
    \item The fibers of $p$ are geometrically integral.
    \item There exists a section $Y \hookrightarrow X$ such that its image $\Theta$ does not intersect any singular point of any singular fiber.
\end{itemize}
In particular, these conditions mean that the generic fiber of $p$ is a smooth elliptic curve, and that the singular fibers are at worst nodal or cuspidal.  We usually write $f$ to denote the class of a fiber of the fibration $p$.

Given a Weierstra{\ss} elliptic fibration, there is always an autoequivalence $\Phi : D^b(X) \to D^b(X)$ corresponding to a moduli problem parametrising rank-one torsion-free sheaves of degree zero on the fibers of $p$.  In fact, $\Phi$ is a relative Fourier-Mukai transform where the kernel is a normalised Poincar\'{e} sheaf.  Since we are only interested in the numerical aspect of the action of $\Phi$ on the Chow ring in this paper, however, we will not say more about the geometric properties of $\Phi$.  The reader may refer to \cite[Section 6.2.3]{FMNT} for more details.

Given a Weierstra{\ss} elliptic fibration $p : X \to Y$, we will simply refer to $X$ as an elliptic surface (resp.\ threefold) when $\dimension X=2$ (resp.\ $3$).


\paragraph[Bridgeland stability conditions]  Given a smooth projective variety $X$, a Bridgeland stability condition $\sigma$ on $D^b(X)$ is a pair $\sigma = (Z,\Ac)$ where $\Ac$ is the heart of a bounded t-structure on $D^b(X)$, and $Z : K(X) \to \mathbb{C}$ is a group homomorphism from the Grothendieck group of $D^b(X)$ to the group of complex numbers under addition, such that $\sigma$ satisfies a positivity property and the Harder-Narasimhan property.  The function $Z$ is often referred to as the central charge in the literature.  We are mainly concerned with computations involving the central charge in this article.  A complete definition of a Bridgeland stability condition can be found in \cite{StabBri}.

\paragraph[For readers without backgrounds in algebraic geometry] \label{para:prelim-nonAG}  The goal of this section is to explain some of  the key terminology from algebraic geometry, so that readers without backgrounds in algebraic geometry can still  make sense of the key computations in this article.

Given a smooth projective variety $X$, its Chow ring $A^\ast (X)$ (over $\mathbb{R}$) is a graded commutative ring 
\[
    A^*(X) = \bigoplus\limits_{i=0}^n A^i(X)
\]
where $n=\dimension X$.  That is, $A^i(X)=0$ whenever $i \notin [0,n]$.  In particular, each $A^i(X)$ is an $\mathbb{R}$-vector space.  Both $A^n(X)$ and $A^0(X)$ are 1-dimensional $\mathbb{R}$-vector spaces and are spanned by the elements $[\mathrm{pt}]$ (the class of a point) and $[X]$ (the fundamental class), respectively; using these bases, we can identify $A^n(X)$ and $A^0(X)$ with $\mathbb{R}$ - these identifications will be used constantly in computations in this article.  The multiplicative identity in $A^\ast (X)$ is the fundamental class $[X]$.

On an algebraic variety $X$, a divisor is an element of $A^1(X)$.  Divisors such as  `ample divisors and `canonical divisors' have special properties in algebraic geometry; we will explicitly state these properties only when needed.

Given an elliptic fibration $p : X \to Y$ where $\dimension X =n$, the assumption that $p$ is a fibration of relative dimension $1$ implies $\dimension Y = n-1$.  That is, when $X$ is a threefold (i.e.\ $\dimension X=3$), $Y$ is a surface (i.e.\ $\dimension Y=2$), whereas when $X$ is a surface, $Y$ is a curve (i.e.\ $\dimension Y=1$).

An elliptic fibration $p : X \to Y$ induces a map $p^\ast : A^\ast (Y) \to A^\ast (X)$ on the Chow rings.  In fact, $p^\ast$ is a homomorphism of graded commutative rings; in particular,  $p^\ast$ preserves multiplication and grading.  For example, the map $p^\ast$ takes the class $[\mathrm{pt}]$  in $A^{n-1}(Y)$ into $A^{n-1}(X)$.  More precisely, since the inverse image of a point $y \in Y$ is a fiber of the fibration $p : X \to Y$, the class $p^\ast [\mathrm{pt}]$ is precisely the fiber class $f$.  Also, the class of the section $\Theta$ lies in $A^1(X)$ and, from the definition of a section, we have $f\Theta = 1$.

Given two elements $Z_1, Z_2$ in some component $A^i(X)$ of  a a Chow ring $A^\ast (X)$, we say $Z_1$ and $Z_2$ are numerically equivalent, and write $Z_1 \equiv Z_2$, if $Z_1W=Z_2W$ for every $W\in A^{n-i}(X)$.


We will also need the following technical result in later sections:

\begin{prop}\cite[Lemmas 3.3, 3.5]{LoWong}\label{lem:LoWong3-3}
Let $A, B \in \mathbb{R}$ with $A \neq 0$.  \begin{itemize}
\item[(i)]  The equation
\begin{equation}\label{eq3}
u=(A+Bu^2)w
\end{equation}
has a power series solution in $u$ (where $u$ is a power series in $w$) given by
\begin{equation}\label{eq5}
u = \sum_{i=0}^\infty c_i A^{i+1}B^i w^{2i+1},
\end{equation}
where $c_i$ is the $i$-th Catalan number.
\item[(ii)] When $A, B \neq 0$, the power series \eqref{eq5} in $w$ has radius of convergence given by $\frac{1}{2\sqrt{|AB|}}$.
\end{itemize}
\end{prop}

\section{Elliptic threefolds and Catalan numbers} \label{sec:Catalanonellip3}
Let $p:X\to Y$ be a Weierstra{\ss} elliptic threefold with $\omega_X \cong \mathcal{O}_X$. We assume there is an ample divisor $H_Y\in A^1(Y)$ that satisfies the numerical equivalence   
    \begin{equation*}
    hH_Y \equiv K_Y   
    \end{equation*}
for some fixed $h \in \mathbb{R}$, where $K_Y$ is the canonical divisor on $Y$. For example, this condition is satisfied by Calabi-Yau Weierstra{\ss} elliptic threefolds over Fano or Enriques surfaces as in \cite[Section 6.6.3]{FMNT} (see \cite[2.1]{Lo15} for other examples).   The above numerical equivalence gives 
\[
\Theta^2 = \Theta p^*K_Y \equiv h\Theta p^*H_Y
\]
where the first equality follows from adjunction \cite[(6.6) and Section 6.2.6]{FMNT}.  In this article, we also assume that the cohomology ring of $X$ over $\mathbb{Q}$ has the decomposition
\[
  H^{2i}(X,\mathbb{Q}) = \Theta p^\ast H^{2i-2}(Y,\mathbb{Q}) \oplus p^\ast H^{2i} (Y,\mathbb{Q})
\]
as in \cite[(6.41)]{FMNT}.

\paragraph[Chern character]
 Under our assumption on the cohomology ring of $X$,  for every object $E \in D^b(X)$ we can write 
    \begin{align}
    &\text{ch}_0(E) = n \notag \\
    &\text{ch}_1(E) = x\Theta + p^*S \notag\\
    &\text{ch}_2(E) = \Theta p^*\eta + af \notag \\
    &\text{ch}_3(E) = s   \label{eq:chE} 
    \end{align}
for some $n, x\in\mathbb{Z}$, $a,s\in\mathbb{Q}$, and $S,\eta\in A^1(Y)$. Let $\Phi$ be the Fourier-Mukai transform as in \ref{para:ellipfib}.  From \cite[Section 6.2.6]{FMNT}, we have  the formula for the cohomological Fourier-Mukai transform
    \begin{align*}
    &\text{ch}_0(\Phi E) = x\\
    &\text{ch}_1(\Phi E) = -n\Theta + p^*(\eta + \tfrac{1}{2}xK_Y)\\
    &\text{ch}_2(\Phi E) = -\Theta p^*(S + \tfrac{1}{2}nK_Y) + (s + \tfrac{1}{2}\eta K_Y + \tfrac{1}{8}xK_Y^2)f - \tfrac{1}{24}xK_Y^2f\\
    &\text{ch}_3(\Phi E) = -(a + \tfrac{1}{2}SK_Y + \tfrac{1}{8}nK_Y^2) - \tfrac{1}{24}nK_Y^2.
    \end{align*}

We can also work out the cohomological Fourier-Mukai transform for twisted Chern characters.  Consider the case where   $\overline{B}, B$ are $\mathbb{R}$-divisors on $X$ satisfying
\begin{equation}\label{eq:Brelation}
\overline{B} = B - \frac{1}{2}p^*K_Y
\end{equation}
 with $B = p^*V$ for some $V\in A^1(Y)$. It is convenient to  set
    \begin{align*}
        & \widehat{S} = S - nV + \tfrac{1}{2}nK_Y \\
        & \widehat{\eta} = \eta - xV + \tfrac{1}{2}xK_Y \\ 
        & \widehat{a} = a - SV + \tfrac{1}{2}nV^2 + \tfrac{1}{2}(S - nV)K_Y + \tfrac{1}{8}nK^2_Y \\ 
        & \widehat{s} = s - \eta V + \tfrac{1}{2}xV^2 + \tfrac{1}{2}(\eta - xV)K_Y + \tfrac{1}{8}xK_Y^2,
    \end{align*}
for then whenever we have
    \begin{align*}
    &\text{ch}_0^{\overline{B}}(E) = n \\
    &\text{ch}_1^{\overline{B}}(E) = x\Theta + p^*\widehat{S} \\
    &\text{ch}_2^{\overline{B}}(E) = \Theta p^*\widehat{\eta} + \widehat{a}f\\
    &\text{ch}_3^{\overline{B}}(E) = \widehat{s}
    \end{align*}
it follows that 
    \begin{align}
    &\text{ch}_0^B(\Phi E) = x \notag\\
    &\text{ch}_1^B(\Phi E) = -n\Theta + p^*\widehat{\eta} \notag\\
    &\text{ch}_2^B(\Phi E) = -\Theta p^*\widehat{S} + \widehat{s}f - \tfrac{1}{24}xK_Y^2f \notag\\
    &\text{ch}_3^B(\Phi E) = -\widehat{a} - \tfrac{1}{24}nK_Y^2, \label{eq:chBPhiE}
    \end{align}
from which it is easier to see the relation between $\ch^{\overline{B}}(E)$ and $\ch^B(\Phi E)$ by comparing \eqref{eq:chE}  and \eqref{eq:chBPhiE}.

\paragraph[Solving a central charge equation]\label{para:solvecce}
For $\omega, B\in A^1(X)$, we define the central charge $Z_{\omega,B}: K(X)\to\mathbb{C}$ by
    \begin{align*}
    Z_{\omega, B}(E) &= -\int_X e^{-(B+i\omega)} \ch(E) \\
    &= -\text{ch}_3^B(E) + \tfrac{\omega^2}{2}\text{ch}_1^B(E) + i(\omega\text{ch}_2^B(E) - \tfrac{\omega^3}{6}\text{ch}_0^B(E)),
    \end{align*}
which is one of the standard central charges used for Bridgeland stability conditions on threefolds \cite{BMT1}.  In order to  solve equations of the form \eqref{eq:centralchargeeq} on an elliptic threefold in which $Z', Z$ are both of the form $Z_{\omega, B}$, and $\Phi$ is the autoequivalence from \ref{para:ellipfib}, however, it turns out that it is more natural to consider `perturbations' of $Z_{\omega, B}$ by introducing an extra parameter as follows: For any $\delta\in\mathbb{R}_{>0}$, we  define the modified central charge
    \begin{equation*}
    Z_{\omega, B}^\delta(E) = -\text{ch}_3^B(E) + \tfrac{\omega^2}{2}\text{ch}_1^B(E) + i(\omega\text{ch}_2^B(E) - \delta\tfrac{\omega^3}{6}\text{ch}_0^B(E)).  
    \end{equation*}


Now suppose $\overline{\omega}, \omega$ are ample divisors on $X$ of the form 
\[
  \overline{\omega} = k\Theta + lp^*H_Y, \text{\quad} \omega = u\Theta + vp^*H_Y
\]
where $k,l,u,v\in\mathbb{R}_{>0}$, and suppose  $\overline{B}, B$ are $\mathbb{R}$-divisors on $X$ as in \eqref{eq:Brelation}.  Then when $\ch(E)$ is of the form \eqref{eq:chE}, for  any $\delta,\epsilon\in\mathbb{R}_{>0}$ we  have  
    \begin{align*}
    Z^\epsilon_{\overline{\omega},\overline{B}}(E) &= -\widehat{s} + \tfrac{1}{2}\overline{\omega}^2\Theta x + \tfrac{1}{2}\overline{\omega}^2p^*\widehat{S} + i\Big( (\overline{\omega}\Theta) p^*\widehat{\eta} +  (\overline{\omega} f)\widehat{a} - \tfrac{\epsilon}{6}\overline{\omega}^3 n \Big), \\
    Z^\delta_{\omega,B}(\Phi E) &= \widehat{a} + (\tfrac{1}{24}K^2_Y - \tfrac{1}{2}\omega^2\Theta)n + \tfrac{1}{2}\omega^2p^*\widehat{\eta} + i\Big(-(\omega\Theta) p^*\widehat{S} + \omega f\widehat{s} - (\tfrac{1}{24}K^2_Y\omega f + \tfrac{\delta}{6}\omega^3)x\Big).
    \end{align*}
Writing 
\[
\gamma = -\left( \tfrac{1}{24}K^2_Y - \tfrac{1}{2}\omega^2\Theta\right), \text{\quad} \gamma' = \tfrac{1}{24}K_Y^2\omega f + \tfrac{\delta}{6}\omega^3,
\]
the above central charge expressions can be rewritten as  
    \begin{align*}
    Z^\epsilon_{\overline{\omega},\overline{B}}(E) &= -\widehat{s} + \tfrac{1}{2}\overline{\omega}^2\Theta x + \tfrac{1}{2}\overline{\omega}^2p^*\widehat{S} + i\Big( (\overline{\omega}\Theta) p^*\widehat{\eta} + \overline{\omega} f \widehat{a} - \tfrac{\epsilon}{6}\overline{\omega}^3 n\Big) \\
    Z^\delta_{\omega,B}(\Phi E) &= \widehat{a} - \gamma n + \tfrac{1}{2}\omega^2p^*\widehat{\eta} + i\Big(-(\omega\Theta) p^*\widehat{S} + \omega f\widehat{s} - \gamma' x\Big).    
    \end{align*}
    
    
Assuming $\gamma, \gamma'>0$ (which holds, for example, when $K_Y=0$ since $\omega$ is ample), we can compare the ratios of coefficients in the two central charges and see that whenever  all the four relations  
     \begin{align}
    \tfrac{6}{\epsilon\overline\omega^3}\cdot \overline{\omega}\Theta \cdot p^\ast \widehat{\eta} & = \tfrac{1}{\gamma}\cdot \tfrac{1}{2}\omega^2 \cdot p^\ast \widehat{\eta} \label{eq:fourconstraints1} \\ 
    \tfrac{6}{\epsilon\overline\omega^3}\cdot \overline{\omega}f & = \tfrac{1}{\gamma} \label{eq:fourconstraints2} \\ 
    \tfrac{1}{\overline{\omega}^2\Theta}\cdot \overline{\omega}^2 & = \tfrac{1}{\gamma'}\cdot\omega\Theta  \label{eq:fourconstraints3} \\ 
   \tfrac{1}{\overline{\omega}^2\Theta}\cdot2 & = \tfrac{1}{\gamma'} \cdot\omega f,  \label{eq:fourconstraints4}
    \end{align}
 hold for  arbitrary $\widehat{\eta}$,  the central charge equation
    \begin{equation}\label{eq:cceq}
    Z^\delta_{\omega,B}(\Phi E) = 
        \begin{pmatrix}
        \tfrac{6\gamma}{\epsilon\overline{\omega}^3} & 0 \\
        0 & \tfrac{2\gamma'}{\overline{\omega}^2\Theta}
        \end{pmatrix}
    (-i) Z^\epsilon_{\overline{\omega},\overline{B}}(E)   
    \end{equation}
 holds for any $E \in D^b(X)$.  

 \begin{rem}
   There are indeed choices of ample divisors $\overline{\omega}, \omega$ and $\epsilon, \delta \in \mathbb{R}_{>0}$ for which the  system of equations \eqref{eq:fourconstraints1}-\eqref{eq:fourconstraints4}  holds for an arbitrary $\widehat{\eta}$  (e.g.\ see \cite{MoralesMT} for a family of solutions).
 \end{rem}

\paragraph[The appearance of Catalan numbers] \label{para:Catalanonellip3} 
Suppose we impose the four relations \eqref{eq:fourconstraints1}-\eqref{eq:fourconstraints4}.  Given \eqref{eq:fourconstraints2}, we can rewrite \eqref{eq:fourconstraints1} as 
\begin{equation}\label{eq:fourconstraints1b}
  2 \frac{\overline{\omega}\Theta}{\overline{\omega}f} p^\ast \widehat{\eta} = \omega^2 p^\ast \widehat{\eta}.
\end{equation}
Since we have
\begin{align*}
  \overline{\omega}\Theta \cdot p^\ast \widehat{\eta} &= (hk+l)\Theta p^\ast(H_Y\widehat{\eta}) \\
  \overline{\omega} f &= k \\
  \omega^2 \cdot p^\ast \widehat{\eta} &= (hu^2+2uv)\Theta p^\ast (H_Y\widehat{\eta}),
\end{align*}
the relation  \eqref{eq:fourconstraints1b} would hold for an arbitrary $\widehat{\eta}$ if we impose the condition 
\begin{equation*}
\tfrac{2(hk+l)}{k}=hu^2+2uv
\end{equation*}
i.e.\
\begin{equation}\label{eq:preCat2}
  u = \left( \tfrac{2(hk+l)}{k}  - hu^2\right) \tfrac{1}{2v}
\end{equation}
which is of the form 
\[
  u = (P+Qu^2)w
\]
if we set 
\[
P =  \tfrac{2(hk+l)}{k}, \text{\quad} Q= -h, \text{\quad} w=\tfrac{1}{2v}.
\]
Assuming $hk+l\neq 0$ (e.g.\ when $l \gg k$, i.e.\ when $\overline{\omega}$ is close to the fiber direction), Proposition \ref{lem:LoWong3-3}(i) now gives
\begin{equation}\label{eq:cat3foldcase}
u = \sum\limits_{i=0}^\infty c_i \left(  \tfrac{2(hk+l)}{k}  \right)^{i+1} (-h)^i \left(\tfrac{1}{2v}\right)^{2i+1}
\end{equation}
where $c_i$ is the $i$-th Catalan number.


In the next section, we show that  Catalan numbers arise from  a more general class of equations in the Chow ring of a fibration.

\section{Fibrations and Catalan numbers} \label{sec:fibandCatalan}

\begin{prop}\label{main:ample_divisor}
Suppose  $p : X \to Y$ is a flat morphism of relative dimension 1, where $X$ and $Y$ are  smooth projective varieties with $X$ being $n$-dimensional, and that $p$ admits a section whose image is denoted by $\Theta$.  Suppose 
\[
  \Theta^2 = h\Theta p^\ast H
\]
for some divisor $H$ on $Y$ and $h \in \mathbb{R}$, and that $\omega$ is a divisor on $X$ of the form 
\[
  \omega = u\Theta + vp^\ast H
\]
for some $u,v \in \mathbb{R}$.  Then for any $0 \neq r \in \mathbb{R}$, the condition
\begin{equation}\label{eq:Catappears}
  \omega^2 (p^\ast \eta) = r(\Theta p^\ast H) (p^\ast \eta) \text{\qquad\qquad for all $\eta \in A^{n-2}(Y)$}
\end{equation}
has a 1-parameter family of solutions in $u,v$ where $u$ can be written as a formal power series  in $\frac{1}{2v}$
\begin{equation*}
  u=\sum_{i=0}^\infty c_i r^{i+1}(-h)^i\left(\tfrac{1}{2v}\right)^{2i+1}
\end{equation*}
where $c_i$ is the $i$-th Catalan number. When $h\neq 0$, as a power series in the variable $\tfrac{1}{2v}$, the above power series has radius of convergence  $\frac{1}{\sqrt{|rh|}}$.
\end{prop}

\begin{proof}
Since 
\[
  \omega^2 = (hu^2 \Theta p^\ast H + 2uv\Theta p^\ast H + v^2 p^\ast H^2), 
\]
for any $\eta \in A^{n-2}(Y)$ we have
\[
  \omega^2 p^\ast \eta = (hu^2 + 2uv)\Theta p^\ast(H\eta).
\]
Hence \eqref{eq:Catappears} can be rewritten as
\[
  (hu^2+2uv)\Theta p^\ast(H\eta) = r\Theta p^\ast(H\eta),
\]
and it holds for an arbitrary $\eta \in A^{n-2}(Y)$ if the relation
\[
  r=hu^2+2uv
\]
holds, i.e.\
\begin{equation}\label{eq:Catappears2}
u=(r-hu^2)\tfrac{1}{2v}.
\end{equation}
By Proposition \ref{lem:LoWong3-3}(i), the equation \eqref{eq:Catappears2} has a power series solution in $u$ of the form
\begin{equation*}
  u = \sum_{i=0}^\infty c_i r^{i+1}(-h)^{i}\left(\tfrac{1}{2v}\right)^{2i+1}
\end{equation*}
where $c_i$ is the $i$-th Catalan number.  Its radius of convergence can be readily computed using  Proposition \ref{lem:LoWong3-3}(ii).
\end{proof}

Proposition \ref{main:ample_divisor} pins down the reason why Catalan numbers appear in solving equations of the form \eqref{eq:stabeq}  on elliptic surfaces as in \cite{LoWong}, and on elliptic threefolds as in \ref{para:Catalanonellip3}.  We describe  the precise connections in the examples below.

\paragraph[Example: elliptic surfaces]  Suppose $p : X \to Y$ is a Weierstra{\ss} elliptic surface in the sense of \cite[Section 6.2.1]{FMNT}.  The dimension of  $X$ is 2 in this case, and if we use $f$ to represent the class of a fiber of $p$, then we have the intersection numbers
\[
 \Theta^2=-e, \text{\qquad} \Theta f = 1, \text{\qquad} f^2=0
\]
for some nonnegative integer $e$.  In this case, for any $\eta \in A^0(X)$ the pullback $p^\ast \eta$ is merely a multiple of the fundamental class of $X$, and so equation \eqref{eq:Catappears} is equivalent to 
\begin{equation}\label{eq:Catellsurf}
\omega^2 = r(\Theta p^\ast H).
\end{equation}
Now take $H$ to be the point class, so that  $p^\ast H$ is simply the fiber class $f$.  Let us also set  
\[
r = 2(m+\alpha -e)
\]
where $m, \alpha$ are real constants as in \cite[10.1]{Lo20} and take $\omega$ to be an ample divisor on $X$ of the form 
\[
\omega = u(\Theta + (m+\tfrac{v}{u})f)
\]
for some positive real numbers $u, v$.  Then \eqref{eq:Catellsurf} becomes
\[
  \tfrac{\omega^2}{2} = m+ \alpha -e
\]
or equivalently
\[
  u = \left( (m+\alpha - e)- (m-\tfrac{e}{2}) u^2 \right) \tfrac{1}{v}
\]
which is exactly the equation \cite[(2.2.1)]{LoWong}.  That is, in the case of Weierstra{\ss} elliptic surfaces, Proposition \ref{main:ample_divisor} reduces to the first half of \cite[Proposition 3.7]{LoWong} which shows that Catalan numbers appear in the process of solving an equation of the form \eqref{eq:stabeq} on elliptic surfaces.

\paragraph[Example: elliptic threefolds]  Suppose $p : X \to Y$ is a $K$-trivial Calabi-Yau Weierstra{\ss} elliptic threefold as defined at the start of Section \ref{sec:Catalanonellip3}.  The dimension of  $X$ is 3 in this case, and by adjunction we have 
\[
  \Theta^2 = \Theta p^\ast K_Y
\]
where $K_Y$ is the canonical divisor of $Y$.  Under our assumptions, there is an ample divisor  $H_Y$  on $Y$, and $K_Y \equiv hH_Y$ for some $h \in \mathbb{R}$.  This gives   $\Theta^2\equiv h\Theta p^\ast H_Y$.  As in \ref{para:solvecce}, let us also assume $\overline{\omega}, \omega$ are ample divisors on $X$ of the form 
\[
  \overline{\omega} = k\Theta + l p^\ast H_Y, \text{\quad} \omega = u\Theta + vp^\ast H_Y
\]
for some $k, l, u, v \in\mathbb{R}_{>0}$.  Then the key equation \eqref{eq:fourconstraints1b}
\[
  2 \frac{\overline{\omega}\Theta}{\overline{\omega}f} p^\ast \widehat{\eta} = \omega^2 p^\ast \widehat{\eta}
\]
becomes
\[
 \omega^2 (p^\ast \widehat{\eta}) =  \tfrac{2(hk+l)}{k} (\Theta p^\ast H_Y) p^\ast \widehat{\eta}.
\]
Applying  Proposition  \ref{main:ample_divisor} with 
\[
r = \tfrac{2(hk+l)}{k}
\]
in \eqref{eq:Catappears} now recovers the power series \eqref{eq:cat3foldcase} containing Catalan numbers, which appeared as part of the process towards solving an equation of the form \eqref{eq:stabeq} on elliptic threefolds.

\bibliography{refs}{}
\bibliographystyle{plain}

\end{document}